\documentclass[12pt]{amsart}

\usepackage[T1]{fontenc}
\usepackage{color}
\usepackage{amssymb,amsmath}
\usepackage{bbm}
\usepackage{paralist}

\usepackage[a4paper,margin=2.5cm]{geometry}

\usepackage[colorlinks,cite color=blue,pagebackref=true,pdftex]{hyperref}

\theoremstyle{plain}
\newtheorem{theorem}{Theorem}[section]
\newtheorem{proposition}[theorem]{Proposition}
\newtheorem{lemma}[theorem]{Lemma}
\newtheorem{corollary}[theorem]{Corollary}
\newtheorem{example}[theorem]{Example}
\newtheorem{quest}[theorem]{Question}

\theoremstyle{definition}
\newtheorem{remark}[theorem]{Remark}

\makeatletter
\CheckCommand*\refstepcounter[1]{\stepcounter{#1}%
  \protected@edef\@currentlabel
  {\csname p@#1\endcsname\csname the#1\endcsname}%
}

\renewcommand*\refstepcounter[1]{\stepcounter{#1}%
  \protected@edef\@currentlabel
  {\csname p@#1\expandafter\endcsname\csname the#1\endcsname}%
}

\def\labelformat#1{\expandafter\def\csname p@#1\endcsname##1}

\DeclareRobustCommand\Ref[1]{\protected@edef\@tempa{\ref{#1}}%
  \expandafter\MakeUppercase\@tempa
}

\makeatother

\newcommand{\ME}{\mathrm{ME}}

\newcommand{\R}{\mathbb{R}}
\newcommand{\C}{\mathbb{C}}
\newcommand{\eins}{\mathbbm{1}}

\newcommand{\Z}{\mathbb{Z}}
\newcommand{\N}{\Z_{\ge0}}
\renewcommand{\P}{\mathbf{P}}

\newcommand{\me}{\mathrm{me}}
\newcommand{\MV}{\mathrm{MV}}
\newcommand{\vol}{\mathrm{vol}}

\renewcommand\emptyset{\varnothing}

\newcommand\Def[1]{\textbf{#1}}

\DeclareMathOperator{\inter}{int}
\DeclareMathOperator{\Ehr}{Ehr}
\DeclareMathOperator{\DMV}{DMV}

\title{Mixed Ehrhart polynomials}
\author[Haase]{Christian Haase}
\thanks{CH was supported by DFG Heisenberg-Professorship HA4383/4}
\address{Christian Haase: Fachbereich Mathematik und Informatik,
  Freie Universit\"at Berlin, Germany}
\email{haase@math.fu-berlin.de}

\author[Juhnke-Kubitzke]{Martina Juhnke-Kubitzke}
\thanks{MJK was supported by the German Research Council DFG-GRK~1916.}
\address{Martina Juhnke-Kubitzke: Institut f\"ur Mathematik, Universit\"at Osnabr\"uck, Germany}
\email{juhnke-kubitzke@uni-osnabrueck.de}

\author[Sanyal]{Raman Sanyal}
\thanks{RS was supported by the DFG-Collaborative Research Center, TRR
  109 ``Discretization in Geometry and Dynamics''.}
\address{Raman Sanyal: FB 12 -- Institut f\"ur Mathematik, Goethe-Universit\"at Frankfurt, Germany}
\email{sanyal@math.uni-frankfurt.de}

\author[Theobald]{Thorsten Theobald}
\thanks{TT was supported by DFG grant TH1333/3-1.}
\address{Thorsten Theobald: FB 12 -- Institut f\"ur Mathematik, Goethe-Universit\"at Frankfurt, Germany}
\email{theobald@math.uni-frankfurt.de}

\date{\today}


\begin{document}

\begin{abstract}
For lattice polytopes $P_1,\ldots, P_k \subseteq \R^d$, Bihan (2016)
introduced the discrete mixed volume $\DMV(P_1,\dots,P_k)$ in analogy to the
classical mixed volume.  In this note we study the associated mixed Ehrhart
polynomial $\ME_{P_1,\dots,P_k}(n) = \DMV(nP_1,\dots,nP_k)$.  We provide a
characterization of all mixed Ehrhart coefficients in terms of the classical
multivariate Ehrhart polynomial. Bihan (2016) showed that the discrete mixed
volume is always non-negative. Our investigations yield simpler proofs for
certain special cases. 

We also introduce and study the associated \emph{mixed $h^*$-vector}.
We show that for large enough dilates $r  P_1, \ldots, rP_k$ the 
corresponding \emph{mixed $h^*$-polynomial} has only real roots  and as a 
consequence  the mixed $h^*$-vector becomes non-negative. 
\end{abstract}

\maketitle


\section{Introduction} \label{sec:intro}

Given a lattice polytope $P \subseteq \R^d$, the number of lattice points
$|P \cap \Z^d|$ is the \Def{discrete volume} of $P$. It is well-known
(\cite{ehrhart-67}; see also \cite{barvinok-2008,BeckRobins})
that the lattice point enumerator
$E_P(n) = |nP \cap \Z^d|$ agrees with a polynomial, the so-called
\Def{Ehrhart polynomial} $E_P(n)\in\mathbb{Q}[n]$ of $P$, for all non-negative
integers $n$.
Recently, Bihan~\cite{bihan-2014} introduced the notion of the \Def{discrete
mixed volume} of $k$ lattice polytopes $P_1, \ldots, P_k \subseteq \R^d$,
\begin{equation} \label{eq:dmv}
  \DMV(P_1, \ldots, P_k) \ := \ \sum_{J \subseteq [k]}
  (-1)^{k-|J|} |P_J \cap \Z^d| \, ,
\end{equation}
where $P_J := \sum_{j \in J} P_j$ for $\emptyset \neq J \subseteq [k]$ and
$P_\emptyset = \{0\}$.

In the present paper, we study the behavior of the discrete mixed
volume under simultaneous dilation of the polytopes $P_i$. This
furnishes the definition of a \Def{mixed Ehrhart polynomial}

\begin{equation}
  \label{eq:mixedehrhart1}
  \ME_{P_1,\ldots,P_k}(n) \ := \ \DMV(nP_1,\dots,nP_k) \ = \
  \sum_{J\subseteq [k]}(-1)^{k-|J|}
  E_{P_J}(n) \ \in \ \mathbb{Q}[n].
\end{equation}

Khovanskii \cite{KhovanskiiGenus} relates the evaluation
$\ME_{P_1,\dots,P_k}(-1)$ to the arithmetic genus of a com\-pac\-ti\-fied
complete intersection with Newton polytopes $P_1, \ldots,
P_k$. Danilov and Khovanskii~\cite{DanilovKhovanskii} investigate the
Hodge-Deligne polynomial $e(Z;u,v) \in \Z[u,v]$ of a complex algebraic
variety $Z$ which in case of a smooth projective variety agrees with
the Hodge polynomial $\sum_{pq} (-1)^{p+q} h^{pq}(Z) u^pv^q$. 
In joint work with Sandra Di Rocco and Benjamin Nill the first author
verified that $\DMV(P_1, \ldots, P_k) = (-1)^{d-k} e(Z;1,0)$ where $Z
\subset (\C^*)^d$ is the uncompactified complete intersection
\cite{haase-comm}. Using the mixed Ehrhart polynomial, this yields the
reciprocity type result
$$ \ME_{P_1,\dots,P_k}(-1) = (-1)^{d-k} e(\bar Z;1,0) \qquad 
\text{ while } \qquad
\ME_{P_1,\dots,P_k}(1) = (-1)^{d-k} e(Z;1,0) \,,
$$
relating $Z$ and its compactification $\bar Z$.

To the best of our knowledge, the origin of the mixed 
Ehrhart \emph{polynomial}~\eqref{eq:mixedehrhart1} goes back 
to Steffens and Theobald \cite{SteffensTheobald}, see also \cite{steffens-diss}.
In their work, a slight variant of~\eqref{eq:mixedehrhart1}, 
yet under the same name, has been employed as a very specific means to
study higher-dimensional, mixed versions of Pick's formula in
connection with the combinatorics of intersections of tropical
hypersurfaces. The definition of the mixed Ehrhart polynomial therein 
differs from our definition by the exclusion of the empty set
from the sum.

Here, we study mixed Ehrhart polynomials and their
coefficients with respect to various bases of the vector space of
polynomials of degree at most $d$.
First, we show that in the usual monomial basis, the coefficients of
mixed Ehrhart polynomials can be read off directly from the
multivariate Ehrhart polynomial $E_{\P}(n_1,\dots,n_k)$, where
$n_1,\ldots,n_k \in \Z_{\ge 0}$ are non-negative integers (Theorem
\ref{thm:ME}). This gives a meaning to the coefficients of
$\ME_{P_1,\dots,P_k}(n)$. In particular this provides a simple
proof that the coefficient of $n^i$ vanishes for $i < k$ and also
allows to give streamlined proofs for known characterizations of the
two leading coefficients (Corollaries~\ref{cor:me_d} and~\ref{cor:me_d-1}).
We then deal with two prominent subclasses.
For the case that all polytopes $P_1, \ldots, P_k$ are equal,
the mixed Ehrhart polynomial can be expressed in terms of the 
$h^*$-vector of $P$ (see Proposition~\ref{prop:singleP}). And for the 
case that $P_1, \ldots, P_k$ all contain the origin and satisfy
$\dim(P_1 + \cdots + P_k) = \dim P_1 + \cdots + \dim P_k$, we can provide
a combinatorial interpretation of $\DMV(P_1, \ldots, P_k)$ 
(see Proposition \ref{prop:M_prod}).

As a consequence of the main result from \cite{bihan-2014}, it follows that
$\ME_{P_1,\dots,P_k}(n) \ge 0$ for all $n \ge 0$ (see
Theorem~\ref{thm:nonneg}).  This is accomplished by the use of irrational
mixed decompositions and skilled estimates. Using our understanding of the
coefficients of $\ME_{P_1,\dots,P_k}(n)$, we give direct proofs for this fact
in the cases $k\in\{2,d-1,d\}$ and $P_1=\cdots=P_k=P$ in
Section~\ref{se:nonneg}. See also~\cite{DMVals} for further
developments in the context of valuations.

Expressing the (usual or mixed) Ehrhart polynomial in the basis $\{
\binom{n+d-i}{d} \, : \, 0 \le i \le d\}$ gives rise to the definition
of the (usual or mixed) $h^*$-vector and $h^*$-polynomial (see
Section~\ref{se:mixedhstar}).  By a famous result of
Stanley~\cite{Stanley-h}, the usual $h^*$-vector is non-negative
(componentwise). We illustrate that for the mixed $h^*$-vector this is
not true in general (see Example~\ref{ex:simpl}). 
Yet, we show that this has to hold asymptotically
for dilates $rP_1, rP_2, \ldots, rP_k$ (Corollary~\ref{cor:proph}). 
This follows from the stronger
result that for $r\gg 0$ the mixed $h^*$-polynomial is real-rooted
with roots converging to the roots of the $d$\textsuperscript{th}
Eulerian polynomial (Theorem~\ref{thm:realRoots}). 
This can be seen as the mixed analogue of
Theorem~5.1 in \cite{DiaconisFulman} (see also
\cite{BeckStapledon,BrentiWelker}).
As a byproduct, we obtain that asymptotically the mixed $h^*$-vector
becomes log-concave, unimodal and, as mentioned, in particular
positive, except for its $0$\textsuperscript{th} entry, which always
equals $0$ (see Corollary~\ref{cor:proph}).

Our paper is structured as follows.  In Section~\ref{se:structure}, we prove
various structural properties of the mixed Ehrhart polynomial. In
Section~\ref{se:nonneg}, we review Bihan's non-negativity result of the
discrete mixed volume, particularly from the viewpoint of the mixed Ehrhart
polynomial, and provide alternative proofs for some special cases. Finally, in
Section~\ref{se:mixedhstar} we study the $h^*$-vector and the $h^*$-polynomial
of the mixed Ehrhart polynomial, and in particular show real-rootedness of the 
mixed $h^*$-polynomial and positivity of the $h^*$-vector for large dilates of 
lattice polytopes.

\section{Structure of the mixed Ehrhart polynomial\label{se:structure}}

\renewcommand\P{\ensuremath{\mathbf{P}}}
In this section, we collect basic properties of the mixed Ehrhart
polynomial. For some known results we provide new or simplified
proofs. For the whole section, we fix a collection $\P =
(P_1,\dots,P_k)$ of $k$ lattice polytopes in $\R^d$ and we assume that
$P_1 + \cdots + P_k$ is of full dimension $d$.  The mixed Ehrhart
polynomial, as introduced in \eqref{eq:mixedehrhart1}, is by
definition a univariate polynomial of degree $\le d$, which can be
written as
\[
    \ME_\P(n) \ = \ \me_d(\P) n^d + \me_{d-1}(\P) n^{d-1} + \cdots + \me_0(\P).
\]

\begin{example}\label{eq:cubes}
    For the case of $k$ copies of the $d$-dimensional unit cube, $P_1 = \cdots
    = P_k = [0,1]^d$, we have $E_{P_i}(n) = (n+1)^d$ and thus
    \[
        \ME_\P(n) \ = \ \sum_{j=0}^k (-1)^{k-j} \binom{k}{j} (jn + 1)^d \ = \
        \sum_{j=0}^k (-1)^{k-j} \binom{k}{j} \sum_{i=0}^d \binom{d}{i} (jn)^i
    \]
    by the binomial theorem. Hence,
    \[
        \ME_\P(n) \ = \ \sum_{i=0}^d \binom{d}{i} n^i \sum_{j=0}^k (-1)^{k-j}
        \binom{k}{j} j^i \, , \\
    \]
    which gives $\me_i(\P) = \binom{d}{i} \Delta^k f(0)$ with $f(x) = x^i$,
    where $\Delta^k f(0)$ denotes the $k$\textsuperscript{th} difference of the function $f$ at
    0 (see, e.g., \cite[Sect.\ 1.4]{Stanley}). This can be expressed as
    $\me_i(\P) = \binom{d}{i} k! S(i,k)$ in the Stirling numbers $S(i,k)$ of
    the second kind (see \cite[Prop.\ 1.4.2]{Stanley}).
\end{example}


\begin{remark}
    If $\P=(P_1,\ldots,P_k)$ is a collection of  lattice polytopes in $\R^d$ 
    and $\dim P_i = 0$ for some $i$, then $\ME_{\P}(n) = 0$.
    To see this, assume that $P_k$ is $0$-dimensional.
    For any $J \subseteq [k-1]$, the polytopes $P_J$ and $P_{J \cup \{k\}}$ are
    translates of each other and the corresponding terms in~\eqref{eq:mixedehrhart1}
    occur with different signs. 
\end{remark}

As our first result, we give a description of the coefficients of the mixed Ehrhart
polynomial in Theorem~\ref{thm:ME}.
 For this we recall a result independently due to Bernstein and
McMullen; see~\cite[Theorem~19.4]{Gruber}.

\begin{theorem}[Bernstein-McMullen]\label{thm:BernMcM}
    For lattice polytopes $\P=(P_1,\dots,P_k)$ in $\R^d$,  the function
    \[
        E_{\P}(n_1,\dots,n_k) \ := \ \bigl| (n_1P_1 + \cdots + n_kP_k) \cap
        \Z^d \bigr|
    \]
    agrees with a multivariate polynomial for all $n_1, \ldots, n_k \in \N$.
    The degree of $E_{\P}$ in $n_i$ is $\dim P_i$.
\end{theorem}

In particular, $E_\P(n_1,\dots,n_k)$ has total degree $d = \dim (P_1 + \cdots +
P_k)$. We write this polynomial as
\begin{equation}\label{eqn:BernMcM}
    E_{\P}(n_1,\dots,n_k) \ = \ \sum_{\alpha \in \N^{k}}
    e_\alpha(\P) \underline{n}^\alpha
\end{equation}
where $\underline{n}^\alpha = n_1^{\alpha_1}\cdots n_k^{\alpha_k}$ for $\alpha=(\alpha_1,\ldots,\alpha_k)\in \N^k$.

\begin{theorem}\label{thm:ME}
    For $\P = (P_1,\dots,P_k)$ a collection of lattice polytopes in $\R^d$ 
    \[
    \me_i(\P) \ = \ \sum_\alpha e_\alpha(\P),
    \]
    where the sum runs over all $\alpha \in \Z_{\ge 1}^k$ such that
    $|\alpha|:=\alpha_1+\cdots+\alpha_k = i$.
    In particular, $\me_i(P) = 0$ for all $0 \le i < k$.
\end{theorem}

\begin{proof}
    For a subset $J \subseteq [k]$ we write $\eins_J \in \{0,1\}^k$ for its
    characteristic vector. Observe that 
    \[
        \DMV(P_1,\dots,P_k) \ = \ \sum_{J \subseteq [k]} (-1)^{k - |J|}
        E_\P(\eins_J).
    \]
  Setting
	\[
        A(\alpha) \ = \ \sum_{J \subseteq [k]} (-1)^{k - |J|} \eins_J^\alpha
        \quad \text{ where } \quad \eins_J^\alpha = 
        \prod_{i = 1}^k
        (\eins_J)_i^{\alpha_i},
    \]
    presentation~\eqref{eqn:BernMcM} then implies
    \[
    \DMV(P_1,\dots,P_k) \ = \ \sum_{\alpha \in \N^k} e_\alpha(\P) A(\alpha).
    \]
    Using that $0^0 = 1$,
    it is easy to verify that $A(\alpha) = 1$ if $\alpha_i > 0$ for all $i
    =1,\dots,k$ and $A(\alpha)=0$ otherwise. Hence,
\[
  \ME_{P_1, \ldots, P_k}(n) = \DMV(nP_1,\dots,nP_k) \ = \ \sum_{\alpha \in
  \Z_{\ge 0}^k} e_{\alpha}(n \P) A(\alpha) \ = \
  \sum_{i=0}^d \big(\sum_{\substack{\alpha \in \Z^k_{\ge 1} \\ |\alpha| = i}} e_{\alpha}(\P) \big) n^i \, ,
\]
where the last step used that $e_{\alpha}$ is homogeneous of degree $(\alpha_1, \ldots, \alpha_d)$.
\end{proof}

The proof also recovers Lemma~4.10 from \cite{SteffensTheobald} with
different techniques.  In order to give an interpretation to some of the
coefficients of $\ME_\P(n)$, let us write
\[
    E_\P(n_1,\dots,n_k) \ = \ \sum_{i=0}^d E_\P^i(n_1,\dots,n_k)
\]
where $E^i_\P(n_1,\dots,n_k)$ is homogeneous in $n_1,\ldots,n_k$ of degree
$i$. By the same reasoning as in the case of the usual Ehrhart polynomial
(see, e.g., \cite[Sect.~3.6]{BeckRobins}), we note
\begin{eqnarray}
    E^d_\P(n_1,\dots,n_k) & = & \vol_d(n_1P_1 + \cdots + n_kP_k) \nonumber \\
     & = & \sum_{\substack{\alpha\in \N^k\\|\alpha|=d}} \tbinom{d}{\alpha_1,\dots,\alpha_k} \MV_d(P_1[\alpha_1],\dots,P_k[\alpha_k]) \underline{n}^\alpha. \label{eq:mixedVolume}
\end{eqnarray}
In the above expression, the (normalized) coefficient $\MV$ is called the
\Def{mixed volume}; see, e.g., \cite[Chapter 6]{Gruber}. Moreover, the
notation $\MV_d(P_1[\alpha_1],\ldots,P_k[\alpha_k])$ means that the polytope
$P_i$ is taken $\alpha_i$ times, and~\eqref{eq:mixedVolume} comes from 
a well-known property of mixed volumes (see \cite[Section 5.1]{Schneider}).
  Based on~\eqref{eq:mixedVolume} together
with Theorem~\ref{thm:ME}, we obtain streamlined proofs for the following
known characterizations of the two highest mixed Ehrhart coefficients (see 
\cite{SteffensTheobald} for the case $k=d-1$ and
\cite[Lemma 3.7 and 3.8]{steffens-diss}).

\begin{corollary}[Steffens, Theobald] \label{cor:me_d}
    Let $\P = (P_1,\dots,P_k)$ be lattice polytopes such that $d = \dim (P_1 +
    \cdots + P_k)$. Then 
    \[
        \me_d(\P) \ = \ \sum_\alpha \tbinom{d}{\alpha_1,\dots,\alpha_k}
        \MV_d(P_1[\alpha_1],\dots,P_k[\alpha_k]),
    \]
    where the sum runs over all $\alpha \in \Z^k_{\ge 1}$ with $|\alpha| = d$.
    In particular, $\ME_\P(n)$ is a polynomial of degree exactly $d$.
\end{corollary}
\begin{proof}
    By Theorem~\ref{thm:ME}, $\me_d(\P)$ is the sum of the coefficients of
    $E_\P(n_1,\dots,n_d)$ of total degree $d$. These are coefficients of
    $E_\P^d(n_1,\ldots,n_k) = \vol(n_1P_1 + \cdots + n_kP_k)$ that are given
    by the mixed volumes. For the second claim, it is sufficient to note that
    since the sum $P_1 + \cdots + P_k$ is full-dimensional,  $E_\P^d$ is not
    identically zero and all coefficients are non-negative.
\end{proof}

For $k = d$, the previous corollary recovers a result of Bernstein~\cite{Bernstein}
\begin{equation}\label{eqn:bern}
    \ME_{P_1,\ldots,P_d}(n) = d!\,\MV_d(P_1,\ldots,P_d)\, n^d.
\end{equation}

A similar reasoning also allows us to give an expression for the second highest 
coefficient $\me_{d-1}(P)$.
For an integral linear functional $a \in (\Z^d)^\vee$, the
rational subspace $a^\perp \subseteq \R^d$ is equipped with a lattice
$a^\perp \cap \Z^d$ and a volume form $\vol_a$ so that a fundamental
parallelpiped has unit volume.
We denote the face of a polytope $P \subseteq \R^d$ where $a$ is maximized
by $P^a$ and consider it (after translation) as a polytope inside
$a^\perp$. Similarly, we write $\MV_a$ ($\DMV_a$) for the (discrete)
mixed volume of a collection of polytopes in $a^\perp$ computed using
$\vol_a$.

\begin{corollary}[Steffens, Theobald]\label{cor:me_d-1}
    Let $\P = (P_1,\dots,P_k)$ be a collection of $d$-dimensional lattice polytopes in $\R^d$. Then
    \[
        \me_{d-1}(\P) \ = \ \frac{1}{2} \sum_{a} \sum_{\substack{\alpha
        \in \Z^k_{\ge 1}\\ |\alpha| = d-1}}
        \binom{d-1}{\alpha_1,\dots,\alpha_k}\MV_{a}(P_1^a[\alpha_1],\dots,
        P_k^a[\alpha_k]),
    \]
    where $a$ ranges over the primitive facet normals of $P_1+\cdots+P_k$.
\end{corollary}

\begin{proof}
For a full-dimensional lattice polytope $P \subseteq \R^d$, the
second highest Ehrhart coefficient can be expressed as
\[
e_{d-1}(P) =  \sum_{a} \frac{1}{2} \vol_{a}(P^a),
\]
where $a$ ranges over all primitive vectors of facets of $P$
(see, e.g., \cite[Thm.~5.6]{BeckRobins}).
This, of course, is a finite
sum as $P$ has only finitely many facets. It follows that
\begin{align*}
  \me_{d-1}(\P) 
  &\ = \ \sum_{J \subseteq [k]} (-1)^{k - |J|} \sum_a
  \frac{1}{2} \vol_{a}(P_J^a)\\
  &\ = \ \frac{1}{2} \sum_a \sum_{J \subseteq [k]} (-1)^{k -
    |J|} \vol_{a}(P_J^a).
\end{align*}
Since $e_d(P)=\vol_d(P)$ for a $d$-dimensional lattice polytope
$P\subseteq \R^d$, \eqref{eq:mixedehrhart1} implies that the inner sum
in the previous expression equals $\me_{d-1}(P_1^a,\dots,P_k^a)$.
The result now follows from Corollary~\ref{cor:me_d} applied to
$\P^a = (P_1^a,\dots,P_k^a)$. Observe that if $a$ is a facet normal for $P_J$, then $a$ is a
facet normal for $P_I$ for all $I \supseteq J$. Hence, the above sum
is over all primitive facet normals of $P_1 + \cdots + P_{k}$.
\end{proof}

Next, we will show that for the special case $P_1 = P_2 = \cdots
= P_k = P$, the mixed Ehrhart polynomial can be expressed in terms of
the $h^*$-vector of $P$; see Section~\ref{se:mixedhstar}. For now, let
us mention that Stanley~\cite{Stanley-h} showed that there exist
non-negative integers $h_i^*(P)$ such that
\begin{equation}\label{eq:hvector}
    E_P(n) \ = \ h_0^*(P) \binom{n+d\,}{d} +  h_1^*(P) \binom{n+d-1}{d}
    + \cdots + h_d^*(P) \binom{n}{d}.
\end{equation}
In terms of this presentation we get:

\begin{proposition}\label{prop:singleP}
    Let $P \subseteq \R^d$ be a $d$-dimensional lattice polytope and let $\P=(P,\ldots,P)$ be a collection 
		of $k$ copies of $P$. Then
		
    \[
        \ME_{\P}(n) \ = \ \sum_{j=0}^d \ \left(\sum_{i=0}^k (-1)^{k-i}\binom{k}{i}\binom{in+d-j}{d}\right)
        h_j^*(P).
    \]
    In particular,
    \[
        \DMV(\P) \ = \ \sum_{j=0}^d \binom{d-j}{d-k}
        h_j^*(P).
    \]
\end{proposition}
\begin{proof}
    From the definition of the mixed Ehrhart polynomial, we infer
    \[
        \ME_{\P}(n) \ = \ \sum_{i=0}^k (-1)^{k-i}\binom{k}{i} E_{iP}(n).
    \]
	Using that $E_{iP}(n)=E_P(in)$ and the expression of $E_P(n)$ as in \eqref{eq:hvector}, this yields
    \[
        \ME_{\P}(n) \ = \ \sum_{j=0}^d h_j^*(P) \sum_{i=0}^k
        (-1)^{k-i}\binom{k}{i}\binom{in+d-j}{d}.
    \]
    This shows the first claim. For the second claim it suffices to
    observe that $\DMV(\P)=\ME_\P(1)$ and to check that, in this case,
    the inner sum in the above expression equals $\binom{d-j}{d-k}$ 
    by the binomial identity in~\cite[(5.24)]{concrete-math}.
\end{proof}

This is reminiscent to the relation between $f$-vectors and $h$-vectors in the
enumerative theory of simplicial polytopes. Investigating this analogy further
yields a theory of discrete mixed valuations; see~\cite{DMVals}.

\begin{example}
    Let $P=[0,1]^3\subseteq \R^3$ be the $3$-dimensional unit cube. 
Then the usual $h^*$-vector of $P$ equals $(1,4,1,0)$. By 
Proposition~\ref{prop:singleP}, 
the discrete mixed volume of the collection $(P,P)$ is given by
\begin{equation*}
\DMV(P,P)=\binom{3-0}{3-2}\cdot 1+\binom{3-1}{3-2}\cdot 4+\binom{3-2}{3-2}\cdot 1=12.
\end{equation*}
%
%
%
The mixed Ehrhart polynomial of $(P,P)$ equals
$\ME_{P,P}(t) =6t^3+6t^2$, which is consistent with Example~\ref{eq:cubes}.


\end{example}


\begin{proposition}\label{prop:M_prod}
    Let $\P = (P_1,\dots,P_k)$ be a collection of lattice polytopes, 
    each containing 0, such that 
    $\dim(P_1 + \cdots +P_k) \ = \ \dim P_1 + \cdots + \dim P_k$.
    Then $\DMV(P_1,\dots,P_k)$ counts the number of lattice points $z \in
    (P_1 + \cdots + P_k) \cap \Z^d$ that are not contained in a subsum
    $P_J$ for $J \neq [k]$.
\end{proposition}

\newcommand\hP{\hat{P}}

\begin{proof}
    By the hypothesis, $P = P_1+\cdots+P_k$ is affinely isomorphic to the
    Cartesian product $P_1 \times \cdots \times P_k$. For any 
    $J \subseteq [k]$, this implies that
    $P_J$ is the intersection of $P$ with the linear span of $P_J$. 
    And for $I, J \subseteq [k]$, we see that
    the intersection $P_I \cap P_J$ is exactly $P_{I \cap J}$. Hence 
    \[
        \DMV(P_1,\dots,P_k) \ = \ |P \cap \Z^d| \ - \ \sum_{J \neq [k]}
        (-1)^{k-|J|} |P_J \cap \Z^d| \ = \ \big| \big(P \setminus
\bigcup_{J\neq [k]} P_J\big)\cap \Z^d \big| \, . \qedhere
\]
\end{proof}

For the case that all polytopes are segments, the
previous proposition can be used to recover~\eqref{eqn:bern}.

\section{Non-negativity of the discrete mixed volume and the mixed
  Ehrhart polynomial\label{se:nonneg}}

Bihan~\cite{bihan-2014} showed the following fundamental non-negativity result. 

\begin{theorem}[Bihan] \label{thm:nonneg}
    Let $P_1,\ldots,P_{k}\subseteq \R^d$ be lattice polytopes. Then
    \[
        \DMV(P_1, \ldots, P_{k}) \ \ge \ 0 \, .
    \]
    Thus the mixed Ehrhart polynomial evaluates to non-negative integers for
    all positive integers $n$.
\end{theorem}

For a proof of Theorem~\ref{thm:nonneg}, Bihan develops a theory of irrational
mixed decompositions which yields a technical induction on dimension.
The discrete mixed volume is a particular combinatorial mixed valuation
in the sense of~\cite{DMVals}. It is shown in~\cite{DMVals} that the discrete
mixed volume is monotone with respect to inclusion, which implies
Theorem~\ref{thm:nonneg}. The proofs in~\cite{DMVals} are less technical but
set in the context of the polytope algebra. 
In this
section we give simple and geometrically sound proofs for the special cases
$k\in\{2,d-1,d\}$ and $P_1=\cdots =P_k=P$.  For the case $k=d$ this is a
consequence of~\eqref{eqn:bern}. For $P_1=\cdots=P_k=P$ this follows from
Proposition~\ref{prop:singleP} together with Stanley's result on the
non-negativity of the $h^*$-vector; see~\cite{Stanley-h,JS15}.

\begin{proof}[Direct proof of Theorem~\ref{thm:nonneg} for $k=2$] 
    Let $P_1$ and $P_2$ be lattice polytopes in $\R^d$. Since the mixed
    Ehrhart polynomial is invariant under translation of the polytopes, we
    may assume that $0$ is a common vertex of $P_1$ and $P_2$, $P_1 \cap P_2 = \{0\}$
    and that there is a
    hyperplane $H$ weakly separating $P_1$ from $P_2$. Hence
    $P_1 \cup P_2 \subseteq P_1 + P_2$. It follows
    that
    \[
        E_{P_1+P_2}(n) \ \geq \ E_{P_1}(n)+E_{P_2}(n)-1,  
    \]
    for all $n \in \Z_{\ge 0}$, i.e., $\ME_{P_1,P_2}(n) \geq 0$ for all $n \in
    \Z_{\ge 0}$.
\end{proof}

For $k=d-1$, a direct proof idea has already been used in the 
framework of tropical geometry and $d$-dimensional Pick-type formulas 
in~\cite{SteffensTheobald}. The following is a proof along similar lines.

\begin{proof}[Direct proof of Theorem~\ref{thm:nonneg} for $k=d-1$] 
    By Theorem~\ref{thm:ME}, the mixed Ehrhart polynomial of a collection of lattice
    polytopes $\P=(P_1,\dots,P_{d-1}) \subseteq \R^d$ is of the form
    \[
        \ME_{\P}(n) \ = \ \me_d(\P) n^d + 
        \me_{d-1}(\P) n^{d-1}.
    \]
    It suffices to show that $\me_d(\P)$ and $\me_{d-1}(\P)$ are non-negative.
    Since the mixed volume is multilinear and symmetric in its entries,
    Corollary~\ref{cor:me_d} yields
    \[
        \me_d(\P) \ = \ \tfrac{d!}{2} \MV_d \left( P_1,P_2,
        \dots, P_{d-1}, \sum_{i=1}^{d-1} P_i \right).
    \]
    For the remaining coefficient, Corollary~\ref{cor:me_d-1} together with the
    fact that $k = d-1$ yields
    \[
        \me_{d-1}(\P) \ = \ \frac{(d-1)!}{2} \sum_a 
        \MV_{a}(P_1^a,\dots,P_{d-1}^a),
    \]
    where the sum is over all primitive facet normals of $P_1 + \cdots +
    P_{d-1}$.
\end{proof}

\section{Mixed $h^*$-polynomials\label{se:mixedhstar}}

For a $d$-dimensional lattice polytope $P$, let
\[
    \Ehr\nolimits_P(z) \ = \ \sum_{n \ge 0} E_P(n) z^n
\]
be the \Def{Ehrhart series} of $P$. As $E_P(n)$ is a polynomial in $n$ of
degree $d$, there exist integers $h_0^*(P), \ldots, h_d^*(P)$ such that 
\[
    \Ehr_P(z) \ = \ \frac{\sum_{j=0}^d h_j^*(P) z^j}{(1-z)^{d+1}};
\]
see, e.g., \cite{barvinok-2008, Gruber}.  The polynomial in the numerator is
called the \Def{$\boldsymbol
h^{\boldsymbol *}$-polynomial} of $P$ (also known as $\delta$-polynomial
\cite{stapledon-2009}) and denoted by $h^*_P(z)$. Similarly, the coefficient
vector $h^*(P):=(h_0^*(P), \ldots, h_d^*(P))$ is called the \Def{$\boldsymbol
h^{\boldsymbol *}$-vector}
of $P$. Stated differently, the Ehrhart polynomial of a $d$-dimensional
lattice polytope $P$ can be expressed as in~(\ref{eq:hvector}).
By Stanley's non-negativity theorem
\cite{Stanley-h}, the coefficients $h_0^*(P), h_1^*(P),\ldots, h_d^*(P)$ are
non-negative integers. Moreover, it is known that
\begin{equation}\label{eq:h}
  h_0^*(P) =  1 \, , \quad
  h_1^*(P) = |P \cap \Z^d| - (d+1) \, , \quad
  h_d^*(P) = |\inter(P) \cap \Z^d|\,,
\end{equation}
where $\inter(P)$ denotes the interior of the polytope $P$. 

The definition of mixed Ehrhart polynomials prompts the notion of a \Def{mixed
$h^*$-vector} $h^*(\P) = (h_0^*(\P),\dots,h_d^*(\P))$ of a collection $\P = (P_1,\dots,P_k)$ of lattice polytopes, which is given by
\begin{equation}\label{eqn:mix_h*}
    \ME_{\P}(n) \ = \ h^*_0(\P) \binom{n+d}{d} + \cdots +
    h^*_d(\P)\binom{n}{d},
\end{equation}
where $d = \dim (P_1+\cdots+P_k)$. In this section, we study properties of
$h^*(\P)$ for a collection of \emph{full-dimensional} lattice polytopes, i.e.,
$\dim P_i = d$ for all $i$. Though Stanley's $h^*$-non-negativity does not
extend to the mixed $h^*$-vector as we will see below, we show that for large
enough dilates $r\P := (rP_1,\dots,rP_k)$, $r \gg 0$, the corresponding
\Def{mixed $\boldsymbol
h^{\boldsymbol *}$-polynomial}
\[
    h_{r\P}^*(z) \ = \ h^*_0(r\P) + \cdots + h_d^*(r\P) z^d 
\]
has only real roots and hence the mixed $h^*$-vector is log-concave. This is
in line with results of Diaconis and Fulman \cite{DiaconisFulman} for the
usual $h^*$-vector. 

\begin{remark} \label{rem:offset}
Note that the contribution of the index set $J= \emptyset$ in the discrete
mixed volume~\eqref{eq:dmv} is $(-1)^k$. For the
{\bf mixed Ehrhart series} $\sum_{n \ge 0} \ME_{P_1, \ldots, P_k}(n) z^n$
this induces for the index set $J = \emptyset$ the contribution
\begin{equation}
  \sum_{n \ge 0} (-1)^k z^n \ = \ (-1)^k \frac{1}{1-z} \ = \ (-1)^k 
    \frac{(1-z)^d}{(1-z)^{d+1}}
  = (-1)^k \frac{\sum_{i=0}^d \binom{d}{i} (-1)^i z^i}{(1-z)^{d+1}}.
\end{equation}
\end{remark}

Since we assume that all polytopes $P_i$ have the same dimension, 
linearity and Remark~\ref{rem:offset} allow to write the mixed $h^*$-vector as
\begin{equation}\label{eq:h*mixed}
    h_i^*(\P) \ = \ \sum_{\emptyset \neq J\subseteq [k]}
    (-1)^{k-|J|} h_i^*(P_J) + (-1)^{k+i} \binom{d}{i}
\end{equation}
for $0\leq i\leq d$.  The next lemma collects some elementary properties of
mixed $h^*$-vectors.

\begin{lemma}
    Let $\P=(P_1,P_2,\ldots,P_k)\subseteq \R^d$ be a collection of $d$-dimensional lattice polytopes. Then:
    \begin{enumerate}[\rm (i)]
    \item $h_0^*(\P)=0$.
    \item If $k=2$, then $h_1^*(P_1,P_2) = \DMV(P_1,P_2)$. In particular, $h_1^*(P_1,P_2) \geq 0$.
\end{enumerate}
\end{lemma}

\begin{proof}
    {(i)} Since, for any lattice polytope $P\subseteq \R^d$ we have
    $h_0^*(P)=1$ (see \eqref{eq:h}), it follows from \eqref{eq:h*mixed} that
    \[  
        h_0^*(\P)=\sum_{\emptyset \neq J\subseteq
        [k]}(-1)^{k-|J|} +(- 1)^k = 0.
    \]
    {(ii)} We know from \eqref{eq:h} that for
    a $d$-dimensional lattice polytope $P\subseteq \mathbb{R}^d$

\begin{equation*}
h_1^*(P)= |P\cap \Z^d|-(d+1)=E_P(1)-(d+1).
\end{equation*}

This fact combined with \eqref{eq:h*mixed} yields

\begin{align*}
h_1^*(P_1,P_2) &= (E_{P_1+P_2}(1) -(d+1))- (E_{P_1}(1)-(d+1) + E_{P_2}(1)-(d+1)) - d \\
&= E_{P_1+P_2}(1)-E_{P_1}(1)-E_{P_2}(1)+1\\
& =\DMV(P_1,P_2),
\end{align*}
where the last equality follows from \eqref{eq:mixedehrhart1}. 
For the second claim, it suffices to note that
$\DMV(P_1,P_2)$ is non-negative (see Section \ref{se:nonneg}).
\end{proof}

A natural question when dealing with integer vectors is whether all entries
are non-negative. The previous lemma provides some positive results in this
direction.
However, as opposed to the $h^*$-vector of a single lattice polytope, the
next example shows that in general it is not true that the coefficients of
the mixed $h^*$-vector are non-negative.

\begin{example}\label{ex:simpl}
  Consider the collection $\P=(\Delta_d,\dots,\Delta_d)$ consisting of $k$
  copies of the standard $d$-simplex. For $(d+1-i) k < d+1$ none of
  the dilates $(d+1-i) P_J = (d+1-i)|J| \Delta_d$ has interior lattice
  points. Hence, $h^*_i(P_J) = 0$ so that $h^*_i(\P) = (-1)^{k+i}
  \binom{d}{i}$ can be negative (see~(\ref{eq:h*mixed})).
  A specific case is $h^*(\Delta_3,\Delta_3) = (0,3,4,-1)$.
\end{example}

Though we have just seen that there do exist collections of polytopes
with negative mixed $h^*$-vector entries, observe that for $m\geq 2$
all entries of $h^*(m\Delta_3,m\Delta_3) = (0,m^3+2m,4m^3,m^3-2m^2)$
are indeed non-negative, and the leading coefficients in $m$ are the
Eulerian numbers $(0,1,4,1)$ (see below). 

We therefore propose to study the following question.

\begin{quest}
Let $\P=(P_1,\ldots,P_k)\subseteq \R^d$ be a collection of
$d$-dimensional lattice polytopes. 
Under which conditions are all (or certain) entries of $h^*(\P)$
non-negative? 
Is it true that $h^*_i(\P) \ge (-1)^{k+i} \binom{d}{i}$?
What can be said if the polytopes are allowed to be of arbitrary
dimension?
\end{quest}


In Corollary \ref{cor:proph} we will show that asymptotically, i.e.,
if one considers high enough dilations of $d$-dimensional polytopes, the mixed
$h^*$-vector $h^*(\P)$ always becomes non-negative.
This suggests that it might
be enough to require the lattice polytopes to contain ``sufficiently
many'' interior points. 

Before we provide our main result of this section, we recall the
definition of the $d$\textsuperscript{th} \Def{Eulerian polynomial} $A_d(z) =
\sum_{k=1}^d A(d,k)z^k$. One, out of several, combinatorial approaches
to define $A_d(z)$ is the following:
\[
    \sum_{n \ge 0} n^d z^n \ = \ \frac{A_d(z)}{(1-z)^{d+1}}.
\]
It is known that the Eulerian polynomials $A_d(z)$ have only simple
and real roots and all roots are negative.  The following result,
which is a generalization of Theorem 5.1 in \cite{DiaconisFulman}, provides
a relation between Eulerian polynomials and mixed $h^*$-polynomials of
lattice polytopes. 

\begin{theorem}\label{thm:realRoots}
    Let $\P = (P_1,\ldots,P_k)$ be a collection of $d$-dimensional lattice
    polytopes. Then, as $r \rightarrow \infty$,
    \begin{equation*}
        \frac{h^*_{r\P}(z)}{r^d} \ \longrightarrow \ \sum_{ J\subseteq
        [k]}(-1)^{k-|J|} \vol(P_J) \, A_d(z).
    \end{equation*}
\end{theorem}

Note that for $k=d$, ~\eqref{eqn:bern} implies the result for $r=1$.

\begin{proof}
    For a $d$-dimensional lattice polytope $P \subseteq \R^d$, the Ehrhart
    polynomial in the usual basis takes the form 
    \[
        E_P(n) \ = \ e_d(P) n^d + e_{d-1}(P) n^{d-1} + \cdots + e_0(P)
    \]
    where $e_d(P) = \vol(P)$. 
		For the Ehrhart series we compute
    \[
    \sum_{n \ge 0} E_P(n) z^n \ = \ 
    \sum_{i=0}^d e_i(P) \sum_{n \ge 0} n^i z^n \ = \ 
    \frac{\sum_{i=0}^d e_i(P) (1-z)^{d-i}A_i(z)}{(1-z)^{d+1}}.
    \]
    By comparing $E_{rP}(n)$ with $E_P(rn)$, we see that $e_i(rP)=r^ie_i(P)$
    and hence
    \[
        h_{rP}^*(z) \ = \ r^d \vol(P) A_d(z) + \sum_{i = 0}^{d-1} r^i e_i(P)
        (1-z)^{d-i}A_i(z).
    \]
    Now from~\eqref{eq:h*mixed} together with the fact that $r P_\emptyset =
    \{0\}$ for all $r>0$, we infer
    \[
    \frac{h^*_{r\P}(z)}{r^d} \ = \ \sum_{J\subseteq [k]} (-1)^{k-|J|}
    \frac{h^*_{rP_J}(z)}{r^d}
    \ \xrightarrow{ \ r \rightarrow \infty \ } \
\sum_{ J\subseteq [k] }(-1)^{k-|J|} \vol(P_J) \, A_d(z),
    \]
    which shows the claim.
\end{proof}

For $k=1$ this also yields the result of Diaconis and Fulman~\cite[Theorem
5.1]{DiaconisFulman} on the asymptotic behavior of the usual $h^*$-polynomial
of a lattice polytope.  Similar results have also been achieved by Brenti and
Welker \cite{BrentiWelker} and Beck and Stapledon \cite{BeckStapledon}.

Before we provide some almost immediate consequences of the previous
result, we recall that a sequence
$(a_0,a_1,\ldots,a_d)$ of real numbers is called \Def{log-concave} if
$a_i^2\geq a_{i-1}a_{i+1}$ for all $1\leq i\leq d-1$. The sequence
$(a_0,a_1,\ldots,a_d)$ is called \Def{unimodal} if there exists a $0\leq
\ell\leq d$ such that $a_0\leq \cdots \leq a_{\ell}\geq \cdots \geq a_d$.

\begin{corollary}\label{cor:proph}
Let $\P=(P_1,P_2,\ldots,P_k)$ be a collection of $d$-dimensional
lattice polytopes in
$\R^d$. Then there exists a positive integer $R$ (depending on
$\P$) such that for $r\geq R$:
\begin{enumerate}[\rm (i)]
\item the mixed $h^*$-polynomial $h^*_{r\P}(z)$ has only real roots 
$\beta^{(1)}(r)<\beta^{(2)}(r)<\cdots
<\beta^{(d-1)}(r)<\beta^{(d)}(r)<0$ with $\lim_{r\rightarrow
  \infty}\beta^{(i)}(r)=\rho^{(i)}$ for $1\leq i\leq d$. 
Here, $\rho^{(1)}<\rho^{(2)}<\cdots < \rho^{(d)}=0$ denote the roots
of $A_d(z)$.
\item $h_i^*(r\P)> 0$ for $1\leq i\leq d$.
\item  $h^*(r\P)$ is log-concave and unimodal.
\end{enumerate}
\end{corollary}

\begin{proof}
\begin{asparaenum}
\item[(i)] First observe that, by Theorem \ref{thm:realRoots}, the roots 
of $h^*_{r\P}(z)$ converge to the roots of $A_d(z)$. Moreover, since the 
roots of $A_d(z)$ are known to be all distinct and negative, there exists 
a positive integer $R$ such that for $r\geq R$, $h^*_{r\P}(z)$ has only simple
and real (negative) roots. Otherwise, since complex roots come up in
pairs (if $\tau$ is a complex root, also its complex conjugate
$\bar{\tau}$ is a root), the Eulerian polynomial $A_d(z)$ would be
forced to have a root of multiplicity $2$, which yields a
contradiction.
\item[(ii)] Using \eqref{eq:mixedehrhart1} and the fact that 
  $e_d(P)=\vol(P)$ for a $d$-dimensional lattice polytope $P\subseteq \R^d$,
  we recognize the right-hand side in Theorem~\ref{thm:realRoots} as 
  $\me_d(\P) A_d(z)$.
  By Corollary~\ref{cor:me_d}, $\me_d(\P)$ is a non-negative linear
  combination of mixed volumes of $P_1,\dots,P_k$ and thus positive. 
   Since the coefficients of the Eulerian polynomials (besides the constant
   term which is 0) are all positive as well,
		Theorem \ref{thm:realRoots} then implies that the sequence
    $\left(\frac{h^*_{r\P}(z)}{r^d}\right)_{r\geq 1}$ converges to a
    polynomial, whose coefficients, except for the constant term (which equals
    $0$), are positive. Hence, there has to exist a positive integer
    $R$ (depending on $\P$) such that for $r\geq R$ all but
    the constant coefficient of $h^*_{r\P}(z)$ are positive. This shows (ii).
\item[(iii)]
    Using the first two parts, we know that there exists a
    positive integer $R$ such that for $r\geq R$ the polynomial
    $h^*_{r\P}(z)$ has, except for the constant term, positive coefficients
    and is real-rooted. Theorem 1.2.1 of \cite{Brenti-PF} implies
    that  $h^*(r\P)$ is log-concave. Since by the choice of $r$,
    this sequence does not have internal zeros,  it is unimodal by
    \cite[Section 2.5]{Brenti-PF}. \qedhere
		\end{asparaenum}
\end{proof}

\medskip

{\bf Acknowledgment.} We thank the referees for very careful reading and their
suggestions which helped to improve the presentation.

\bibliography{citations}

\begin{thebibliography}{10}

\bibitem{haase-comm}


\bibitem{barvinok-2008}
A.~Barvinok.
\newblock {\em Integer Points in Polyhedra}.
\newblock Zurich Lectures in Advanced Mathematics. European Mathematical
  Society (EMS), Z\"urich, 2008.

\bibitem{BeckRobins}
M.~Beck and S.~Robins.
\newblock {\em Computing the Continuous Discretely}.
\newblock Springer, New York, 2007.

\bibitem{BeckStapledon}
M.~Beck and A.~Stapledon.
\newblock On the log-concavity of {H}ilbert series of {V}eronese subrings and
  {E}hrhart series.
\newblock {\em Math. Z.}, 264(1):195--207, 2010.

\bibitem{Bernstein}
D.~N. Bernstein.
\newblock The number of roots of a system of equations.
\newblock {\em Funkcional. Anal. i Prilo\v zen.}, 9(3):1--4, 1975.

\bibitem{bihan-2014}
F.~Bihan.
\newblock Irrational mixed decomposition and sharp fewnomial bounds for
  tropical polynomial systems.
\newblock {\em Discrete Comput.\ Geom.}, 55(4):907--933, 2016.

\bibitem{Brenti-PF}
F.~Brenti.
\newblock Unimodal, log-concave and {P}\'olya frequency sequences in
  combinatorics.
\newblock {\em Mem. Amer. Math. Soc.}, 81(413), 1989.

\bibitem{BrentiWelker}
F.~Brenti and V.~Welker.
\newblock The {V}eronese construction for formal power series and graded
  algebras.
\newblock {\em Adv. in Appl. Math.}, 42(4):545--556, 2009.

\bibitem{DanilovKhovanskii}
V.~I. Danilov and A.~G. Khovanskii.
\newblock Newton polyhedra and an algorithm for computing {Hodge--Deligne}
  numbers.
\newblock {\em Math.\ USSR Izvestiya}, 29(2):279--298, 1987.

\bibitem{DiaconisFulman}
P.~Diaconis and J.~Fulman.
\newblock Carries, shuffling, and symmetric functions.
\newblock {\em Adv. in Appl. Math.}, 43(2):176--196, 2009.

\bibitem{ehrhart-67}
E.~Ehrhart.
\newblock Sur un probl\`eme de g\'eom\'etrie diophantienne lin\'eaire. {I}.
  {P}oly\`edres et r\'eseaux.
\newblock {\em J. Reine Angew. Math.}, 226:1--29, 1967.

\bibitem{concrete-math}
R.L. Graham, D.E. Knuth, and O.~Patashnik.
\newblock {\em Concrete Mathematics}.
\newblock Addison-Wesley, Reading, MA, 1990.

\bibitem{Gruber}
P.~M. Gruber.
\newblock {\em Convex and Discrete Geometry}.
\newblock Springer, Berlin, 2007.

\bibitem{DMVals}
K.~Jochemko and R.~Sanyal.
\newblock Combinatorial mixed valuations.
\newblock Preprint, May 2016, 16 pages, \href{http://arxiv.org/abs/
  1605.07431}{arXiv:1605.07431}.

\bibitem{JS15}
K.~Jochemko and R.~Sanyal.
\newblock Combinatorial positivity of translation-invariant valuations and a
  discrete {H}adwiger theorem.
\newblock To appear in \emph{J.\ Eur.\ Math.\ Soc.}, 2015.

\bibitem{KhovanskiiGenus}
A.~G. Khovanskii.
\newblock Newton polyhedra, and the genus of complete intersections.
\newblock {\em Funktsional.\ Anal.\ i Prilozhen}, 12(1):51--61, 1978.

\bibitem{Schneider}
R.~Schneider.
\newblock {\em Convex {B}odies: {T}he {B}runn-{M}inkowski {T}heory}, volume 151
  of {\em Encyclopedia of Mathematics and its Applications}.
\newblock Cambridge University Press, Cambridge, 2nd edition, 2014.

\bibitem{Stanley-h}
R.~P. Stanley.
\newblock Decompositions of rational convex polytopes.
\newblock {\em Ann.\ Discrete Math.}, 6:333--342, 1980.

\bibitem{Stanley}
R.~P. Stanley.
\newblock {\em Enumerative Combinatorics. {V}ol.\ 1}, volume~49 of {\em
  Cambridge Studies in Advanced Mathematics}.
\newblock Cambridge University Press, Cambridge, second edition, 2012.

\bibitem{stapledon-2009}
A.~Stapledon.
\newblock Inequalities and {E}hrhart {$\delta$}-vectors.
\newblock {\em Trans. Amer. Math. Soc.}, 361(10):5615--5626, 2009.

\bibitem{steffens-diss}
R.~Steffens.
\newblock {\em Mixed Volumes, Mixed {E}hrhart Theory and Applications to
  Tropical Geometry and Linkage Configurations}.
\newblock PhD thesis, Goethe-Universit\"at Frankfurt, 2009.

\bibitem{SteffensTheobald}
R.~Steffens and T.~Theobald.
\newblock Combinatorics and genus of tropical intersections and {E}hrhart
  theory.
\newblock {\em SIAM J. Discrete Math.}, 24(1):17--32, 2010.

\end{thebibliography}
\bibliographystyle{plain}

\end{document}